\documentclass{svproc}

	\usepackage[utf8]{inputenc}
	\usepackage[T1]{fontenc}
	\usepackage[english]{babel}
	\usepackage{amsmath, amsfonts, amssymb}
	\usepackage{bm, enumerate, graphicx, array, xcolor, tikz}
	\usepackage{cite, multirow, multicol, footmisc, type1cm}

 \newcommand{\fn}{\mathfrak{n}}
 \newcommand{\fr}{\mathfrak{r}}
 \newcommand{\fa}{\mathfrak{a}}
 \newcommand{\fh}{\mathfrak{h}}
 \newcommand{\fd}{\mathfrak{d}}
 \newcommand{\fg}{\mathfrak{g}}
 \newcommand{\fl}{\mathfrak{l}}
 \newcommand{\fgl}{\mathfrak{gl}}
 \newcommand{\fsl}{\mathfrak{sl}}
 \newcommand{\fs}{\mathfrak{s}}
 \newcommand{\fv}{\mathfrak{v}}
 \newcommand{\fsu}{\mathfrak{su}}
 \newcommand{\ku}{\mathbb{K}}
 \newcommand\blank{{\mkern 2mu\cdot\mkern 2mu}}

 \DeclareMathOperator{\cdim}{codim}
 \DeclareMathOperator{\ad}{ad}
\DeclareMathOperator{\der}{Der}
 \DeclareMathOperator{\spa}{span} 
 \DeclareMathOperator{\inner}{Inner} 
 \DeclareMathOperator{\tr}{tr}

\title{Examples and patterns on quadratic Lie algebras}
\author{Pilar Benito\inst{1} \and Jorge Rold\'an-L\'opez\inst{2}}
\institute{Universidad de La Rioja, Departamento de Matem\'aticas y Computaci\'on, 
Complejo Cient\'ifico Tecnol\'ogico, Calle Madre de Dios 53, 26006 Logro\~no, La Rioja, Spain \email{pilar.benito@unirioja.es}
\and 
Universidad de La Rioja, Departamento de Matem\'aticas y Computaci\'on, Complejo Cient\'ifico Tecnol\'ogico, Calle Madre de Dios 53, 26006 Logro\~no, La Rioja, Spain
\email{jorge.roldanl@unirioja.es}}

\begin{document}
\maketitle

\begin{center}
\emph{In the memory of our collegue and friend Georgia Benkart.}
\end{center}

\begin{abstract}
    A Lie algebra is said to be quadratic if it admits a symmetric invariant and non-degenerated bilinear form. Semisimple algebras with the Killing form are examples of these algebras, while orthogonal subspaces provide abelian quadatric algebras. The class of quadratic algebras is outsize, but at first sight it is not clear weather an algebra is quadratic. Some necessary structural conditions appear due to the existence of an invariant form forces elemental patterns. Along the paper we overview classical features and constructions on this topic and focus on the existence and constructions of local quadratic.
\end{abstract}
    
\section{Introduction}

A bilinear form $\varphi$ on a Lie algebra $\fg$ with product $[x,y]$ that satisfy,
    \begin{equation}\label{eq:quadratic}
    \varphi([x,y],z)=\varphi(x,[y,z]) \quad\forall x,y,z\in \fg
    \end{equation}
is said invariant. In addition if $\varphi$ is symmetric and non-degenerated, the pair $(\fg, \varphi)$ is named quadratric Lie algebra (some authors called them quasi-classical, or self-dual, or even orthogonal). 
Finite-dimensional real quadrartic algebras appear as metric or metrizable. They were introduced in 1957 by S. T. Tsou and A. G. Walker \cite{Tsou_Walker_1957} who provided basic well known facts on existence and structure of this class of algebras. Tsou and Walker were looking for pretty properties to distinguish metrizable from non-metrizable Lie groups. The structure results they published are a compendium of the Tsou's PhD Thesis (Liverpool, 1955) whose adviser was Walker.

 In mid-1980s and along 1990s, two different methods of construction of quadratic Lie algebra appears: the double extension process and the $T^*$-extension (open up to general algebras in characteristic different from $2$). 
 The first method is a multi-step process that was introduced independently by several authors (Kac, Keith, Hoffmann, Favre, Santharoubane, Medina and Revoy). According to \cite{medina1985algebres}, any non-simple non-abelian indecomposable quadratic Lie algebra is a double extension of either by a one-dimensional or by a simple algebra. The $T^*$-extension is a one-step method that produces quadratic algebras as central extensions of any arbitrary $\fg$ Lie algebra by using a cyclic $2$-cocycle $\omega \in Z^2(\fg,\fg^*)$ (see \cite{bordemann1997nondegenerate} for definitions). The most basic example is the null extension $T^*_0(\fg)=\fg\oplus \fg^*$, so $\omega=0$, with the hyperbolic form $q(a+\beta, a'+\beta')=\beta'(a)+\beta(a')$ and product
    \begin{equation}\label{eq:T*-extension}
    [a+\beta, a'+\beta']=[a,a']_{\fg}+\beta\circ \ad\, a' -\beta'\circ\ad a.
    \end{equation}
For arbitrary $T^*_\omega$, the term $\omega(a,a')$ must be added to the expression \eqref{eq:T*-extension}.

The double extension procedure is a bit more complex. We start with a quadratic Lie algebra $(\fa, \varphi_\fa)$ and another Lie algebra $\fg$ such that there exists a Lie homomorphism $\rho\colon \fg \to \der_{\varphi_\fa} \fa$. From $\rho$ we get the $2$-cocycle $\omega\colon \fa\times\fa\to \fg^*$ given by $\omega(a,b)(x)=\varphi_\fa(\rho(x)(a),b)$ that let us define the central extension $\fa \oplus_\omega \fg^*$, under the multiplication
\begin{equation*}
    [a+\beta, a'+\beta']=[a,a']_\fa+\omega(a,b).
\end{equation*}
Now, from the coadjoint representation $\ad^*$ of $\fg$ ($\ad^*(x)(\beta)=-\beta\circ\ad x$) and $\rho$, we reach the homomorphism of Lie algebras $\psi\colon \fg\to \der (\fa\oplus_\omega \fg^*)$, $\psi(x)(a+\beta)=\rho(x)(a)-\beta\circ\ad x$. The semidirect product $\fa(\fg)=\fg \oplus_\psi(\fa\oplus_\omega \fg^*)$ is a Lie algebra with bracket
\begin{multline}\label{eq:de-bracket}
		[b + a + \beta, b' + a' + \beta'] :=  [b,b']_\fg + \rho(b)(a') - \rho(b')(a) + [a,a']_\fa \\+ w(a, a') + \beta'\circ\ad\,b - \beta\circ\ad\,b',
	\end{multline}
	and the bilinear form $\varphi$ is non-degenerated and symmetric on $\fa(\fg)$:
	\begin{equation}\label{eq:de-bilinear}
		\varphi(b + a + \beta, b' + a' + \beta') := \beta(b') + \beta'(b) + \varphi_{\fa}(a, a').
	\end{equation}
So, the pair $(\fa(\fg), \varphi)$ is a quadratic Lie algebra over $\mathbb{K}$ and it is called the double extension of $(\fa,\varphi_\fa)$ by $(\fg, \rho)$.

Along the 2000s, many efforts were made on small dimension classifications of quadratic Lie algebras or in deepening the structure of certain families related to symmetric spaces (Kath, Olbrich, Duong, Benayadi, Hilgert, Neeb). From then until today (see \cite{ovando2016lie} and references therein), we find papers on classification by using non-classical procedures (Benito, Laliena, de-la-Concepción) or relating quadratics algebras, geometric structures and applications (Rodriguez-Vallarte, Salgado, Cornulier, del Barco, Conti, Rossi).

Throughout this paper we will summarise features and results on quadratic algebras in three sections apart from the Introduction. Section 2 includes basic terminology, existence results and duality. Local quadratic algebras, their structure and constructions, are treated in Section 3. The final section is devoted to talk about the 2014 Elduque and Benayadi contribution  \cite{benayadi2014classification} on mixed quadratic as a tribute to Alberto on his 60th birthday.

We restrict to finite-dimensional algebras over a field $\ku$ of characteristic zero. These results, collected in different articles, have been revisited, expanded and exemplified along this paper.

\section{Patterns and duality}

Along this section $(\fg, \varphi)$ is a quadratic Lie $\ku$-algebra with product $[x,y]$. In general, for $U$ and $V$ subspaces of $\fg$, $U^\perp=\{x\in \fg: \varphi(x,u)=0\ \forall u\in U\}$ and  $[U,V]=\spa\langle[u,v]:u\in U, v\in V\rangle$. 

The \emph{lower (upper) central series} of $\fg$ is inductively defined as $\fg^1=\fg$ and $\fg^{t+1}=[\fg,\fg^t]$ ($Z_0(\fg)=\{0\}$, $Z_1=Z(\fg)$ the centre and $Z_{t+1}=\{x\in \fg: [x,\fg]\subset Z_t\}$). And the terms in \emph{derived series} are $\fg^{(1)}=\fg$ and $\fg^{(t+1)}=[\fg^{(t)},\fg^{(t)}]$. A Lie algebra $\fg$ is called \emph{nilpotent} (\emph{solvable}) if $\fg^t = 0$ for some $t \geq 1$ ($\fg^{(t)} = 0$ for some $t \geq 1$). The left multiplication by $x\in \fg$, $\ad x=[x,\blank]$, is a derivation known as \emph{inner derivation}. In general, a derivation $d$ of $\fg$ such that $\varphi(d(x),y)+\varphi(x,d(y)) = 0$ will be called \emph{$\varphi$-skew symmetric} and $\der_\varphi \fg$ will denote the set of skew-symmetric derivations with respect to $\varphi$, while $\der \fg$ and $\inner \fg$ will denote the whole set of derivations and inner derivations respectively. We note that condition (\ref{eq:quadratic}) is equivalent to saying that $\inner \fg$ is a subalgebra of $\der_\varphi \fg$ (in fact, it is an ideal). The algebra $\fg$ is \emph{reduced} if its centre is contained in the derived algebra $\fg^2$. And, $\fg$ is called \emph{decomposable} if it contains a proper and non-degenerate ideal $I$ (so $I\cap I^\perp=0$). Otherwise, we say that $\fg$ is \emph{indecomposable}.

Semisimple Lie algebras by means of their Killing form are quadratic. Abelian quadratic are just orthogonal vector spaces $(\fv, \varphi)$, $[u,v]=0$ for all $u,v\in \fv$. Orthogonal direct sums of quadratic algebras (as ideals) produce new quadratic algebras. Therefore, any \emph{reductive} Lie algebra $\fg=\fs\oplus Z(\fg)$, $\fs$ semisimple, is quadratic. But the class of quadratic algebras is very large and assorted. There are a wide variety that includes non-abelian nilpotent, solvable (non-nilpotent) and mixed (non-solvable and non-semisimple). In fact, denoting by $(r,s)=(\dim \,\fg^2,  \dim \,Z(\fg))$, in \cite[Theorem 5.1]{tsou1962xi} we find the following result:

\begin{theorem}[Tsou, 1962]
    There exist reduced quadratic Lie algebras of arbitrary type $(r,s)$ except for $(5, 0)$,  $(7,0)$ and $(4,s)$, $0\leq s\leq 4$.
\end{theorem}

\noindent The proof of this theorem, based on multilinear arguments and tools, does not make it clear how to build them.

Since $\varphi$ is non-degenerated and invariant, the centre and the derived algebra of $\fg$ are orthogonal:
\begin{equation}\label{eq:Z-der-orto}
    (\fg^2)^\perp=\{x\in\fg:\varphi([xy],z)=0=-\varphi (y,[xz])\ \forall x,y \in \fg \}=Z(\fg), 
\end{equation}and therefore
\begin{equation}\label{eq:Z-der-dimens}
    \dim \fg=\dim (\fg^2)^\perp+\dim Z(\fg). 
\end{equation}

\noindent Equalities \eqref{eq:Z-der-orto} and \eqref{eq:Z-der-dimens} are two of the most famous patterns on quadratic algebras. The following proposition covers other well-known features.

\begin{proposition}\label{prop:pattern-1}
Let $(\fg, \varphi)$ a quadratic Lie algebra, $U$ any subspace and $I$ an ideal. Then:

\begin{enumerate}[\quad a)]
    
    \item $\dim \fg=\dim U+\dim U^\perp$ and, if $U$ is non-degenerate, $\fg=U\oplus U^\perp$.
    
    \item $U$ is an ideal of $\fg$ if and only if $[U^\perp, U]=0$.
    
    \item Any minimal and non-degenerate ideal of $\fg$ is simple or one-dimensional. Minimal degenerated ideals are isotropic and abelian.
   
    \item The algebra $\fg$ decomposes as the orthogonal direct sum, as ideals, of a reduced quadratic and an abelian.
   
    \item If $\fg$ is indecomposable $I\cap I^\perp\neq 0$ and $Z(\mathfrak{g})\subseteq \mathfrak{g}^2$. In particular, $\fg$ is a reduced algebra and has not simple ideals.
    
    \item The centre of a nonzero solvable quadratic algebra is nonzero.
    
    \item If $\fg$ is indecomposable, $\fg$ is one-dimensional or simple or isometrically isomorphic to a double extension of an algebra $\fh$ one-dimensional or simple.

\end{enumerate}
\end{proposition}
\begin{proof}
    A detailed proof of items b) to c) and g) can be found in \cite{figueroa1996structure}. Item d) is just \cite[Proposition 6.2]{Tsou_Walker_1957} and e) follows easily from previous items. For item f), apply \eqref{eq:Z-der-orto} and $\fg^2\neq \fg$ because of $\fg$ is nonzero and solvable.
\end{proof}

We introduce examples on quadratic and non quadratic algebras regarding previous Proposition \ref{prop:pattern-1} and classical constructions.

\begin{example}\label{ex:oscillator}
The smallest solvable non-abelian quadratic algebra is the oscillator $(\fd_4, \psi)$ with basis $x_1,x_2,x_3,z$, nonzero products $[x_1,x_2]=x_3$, $[x_1,x_3]=-x_2$, $[x_2,x_3]=z$ and form $\psi(x_1,z)=1$ and $\psi(x_2,x_2)=\psi(x_3,x_3)=1$. The derived algebra $\fd_4^2=[\fd_4,\fd_4]=\fh_1$. An easy computation yields $\der_\psi \fd_4=\inner \fd_4$. This algebra arises over the reals in the quantum mechanical description of a harmonic oscillator and $\psi$ is a Lorentzian form.  \hfill $\square$
\end{example}

\begin{example}\label{ex:tensor_product}
Following \cite[Proposition A]{hofmann1986invariant}, the tensor product $(\fg\otimes \fa, \varphi\otimes q)$ where $(\fg, \varphi)$ is quadratic Lie with bracket $[x,y]$ and $(\fa, q)$ an associative and commutative algebra with invariant symmetric and nondegenerate form $q$ and product $ab$ produces the quadratic algebra $[x\otimes a,y\otimes b]=[x,y]\otimes ab$ under the form $(\varphi\otimes q)(x\otimes a,y\otimes b)=\varphi(x,y)q(a,b)$. As a remarkable example we point out $(\fs\otimes \frac{\ku[t]}{\langle t^n\rangle}, \mathcal{K}\otimes q)$ for any $n\geq 1$, $\fs$ semisimple, $K(x,y)=\tr (\ad x\circ \ad y)$ the Killing form and $q(x^i,x^j)=1$ ($q(x^i,x^j)=0$) if and only if $i+j+1=n$ (otherwise), here $x$ is the class of $t$ module $\langle t^n\rangle$.  \hfill $\square$
\end{example}

\begin{example}
The generalized Heisenberg algebra series is determined through the property $\fg^2=Z(\fg)=\spa\langle z\rangle$. They are $2$-nilpotent ($\fg^3=0$) of odd dimension. For any $n\geq 1$ there exists a unique $n^{th}$ Heisenberg algebra of dimension $2n+1$ that we denote from now on as $\fh_n$. The algebra $\fh_n$ has a \emph{standard basis} $e_1,\dots, e_n,e_{n+1}, \dots e_{2n}, z$ with nonzero brackets $[e_i,e_{n+1}]=z$. Note that item a) in Proposition \ref{prop:pattern-1} fails, so $\fh_n$ are not quadratic. Doubling by the dual space $(\fh_n)^*$, we get the trivial quadratic $T^*$-extension $T_0(\fh_n)$ as it is defined in \eqref{eq:T*-extension}. In \cite[Example 4.1]{bordemann1997nondegenerate}, examples of nilpotent quadratic with arbitrary nilindex follows from the non-quadratic nilpotent filiform algebra series $\fl_n$, for $n\geq 2$. \hfill $\square$ 
\end{example}

\begin{example}\label{ex:ndt-free}
According to \cite{del2012free}, the $5$-dimensional algebra $\fn_{2,3}$ and the $6$-dimensional $\fn_{3,2}$ are the unique quadratic free nilpotent Lie algebras. Any algebra in the free nilpotent series $\fn_{d,t}$ ($d\geq 2$ and $t\geq 2$) satisfies: the nilindex is $t+1$, $d=\cdim \fn_{d,t}^2$ and the centre, $Z(\fn_{d,t})=(\fn_{d,t})^t$, has dimension $\frac{1}{t}\sum_{a|t}\mu(a)d^{t/a}$, where $\mu$ the is M\"obius function (see \cite{gauger_1973}). Then, the assertion is a corollary of item a) in Proposition \ref{prop:pattern-1}. The quadratic structure of $\fn_{2,3}$ is given by the (Hall) basis $\{a_i\}_{i=1}^5$ and nonzero products $[a_2 , a_1] = a_4$, $[a_3 , a_1] = a_5$, $[a_3 , a_2] = a_6$ and the symmetric form $\varphi(a_i, a_j) = (-1)^{i-1}$ for $i \leq j$ and $i + j = 6$ and $\varphi(a_i, a_j) = 0$ otherwise. For $\fn_{3,2}$, take as (Hall) basis $\{a_i\}_{i=1}^6$ and nonzero products $[a_2 , a_1] = a_4$, $[a_3 , a_1] = a_5$, $[a_3 , a_2] = a_6$ and symmetric form $\phi(a_i, a_j) = (-1)^{i-1}$ for $i \leq j$ and $i + j = 7$ and $\phi(a_i, a_j) =0$ otherwise.  \hfill $\square$
\end{example}

The existence of an invariant non-degenerated bilinear form on $\fg$ is equivalent to saying that the adjoint $\rho=\ad$ and coadjoint $\rho^*$ representations of $\fg$ are isomorphic. This pattern has a significant influence on the lattice of ideals of quadratic algebras (see \cite{hofmann1986invariant} and \cite{keith1984invariant}).

\begin{proposition}[Keith, Hofmann, 1984]
    For a quadratic Lie algebra $(\fg, \varphi)$, the map $\Omega\colon \text{Set of ideals of} \ \fg\to \text{Set of ideals of} \ \fg$, $\Omega(I)=I^\perp$ is an involutive anti-automorphism of the lattice of ideals of $\fg$. In particular lattices of ideals of quadratic Lie algebras are self-dual and, 

    \begin{enumerate}[\quad a)]
        \item $I$ is a minimal ideal if and only if $I^\perp$ is maximal. So, the sum of the set of minimal ideals, socle ideal, and the intersection of the set of maximal ideals, Jacobson radical, are orthogonal.
        \item The terms $\fg^{t+1}$ and $Z_t(\fg)$ of the lower and upper central series are orthogonal and $\dim\ \fg =\dim\ \fg^{t+1} + \mathrm{dim}\ Z_t(\fg)$.
    \end{enumerate}
\end{proposition}

The lattice of ideals of the oscillator $\fd_4$ is a $4$-chain in case $x^2+1$ is irreducible over $\ku$. Otherwise, there are eigenvectors of $\ad x_1$, $[x_1,u]=\lambda u$ and $[x_1,v]=-\lambda v$ and $Z(\fd_4)=\langle[u,v]\rangle$ and the proper ideals of $\fd_4$ are $\fd_4^2$, $\langle u,[u,v]\rangle$, $\langle v,[u,v]\rangle$ and the centre. Both lattices are self-dual. The next figure shows other lattices of quadratic Lie algebras.

\newcommand{\hd}{5}   
	\newcommand{\dnp}{-6} 
	\newcommand{\dsp}{-2} 
	\newcommand{\drp}{2}  
	\newcommand{\dbp}{6}  
	\newcommand{\anchCubo}{\hd/3}
	\newcommand{\altoCubo}{\hd/3}
	\newcommand{\sepYCubo}{\hd/3}
	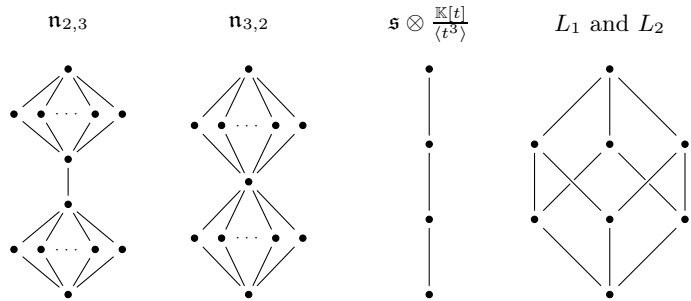
\begin{figure}[h]
	
		\centering
		\begin{tikzpicture}[scale=.6]
			\node (DN6)  at (\dnp,    \hd) {};
			\node (DN55) at (\dnp+1.2,\hd*4/5) {};
			\node (DN54) at (\dnp+0.6,\hd*4/5) {};
			\node (DN53) at (\dnp,    \hd*4/5) {\tiny$\cdots$};
			\node (DN52) at (\dnp-0.6,\hd*4/5) {};
			\node (DN51) at (\dnp-1.2,\hd*4/5) {};
			\node (DN4)  at (\dnp,    \hd*3/5) {};
			\node (DN3)  at (\dnp,    \hd*2/5) {};
			\node (DN25) at (\dnp+1.2,\hd*1/5) {};
			\node (DN24) at (\dnp+0.6,\hd*1/5) {};
			\node (DN23) at (\dnp,    \hd*1/5) {\tiny$\cdots$};
			\node (DN22) at (\dnp-0.6,\hd*1/5) {};
			\node (DN21) at (\dnp-1.2,\hd*1/5) {};
			\node (DN1)  at (\dnp,    0) {};
			\node (DNN)  at (\dnp,    \hd+1) {$\mathfrak{n}_{2,3}$};
			
			\filldraw [black] (DN6) circle (2pt);
			\filldraw [black] (DN55) circle (2pt);
			\filldraw [black] (DN54) circle (2pt);
			\filldraw [black] (DN52) circle (2pt);
			\filldraw [black] (DN51) circle (2pt);
			\filldraw [black] (DN4) circle (2pt);
			\filldraw [black] (DN3) circle (2pt);
			\filldraw [black] (DN25) circle (2pt);
			\filldraw [black] (DN24) circle (2pt);
			\filldraw [black] (DN22) circle (2pt);
			\filldraw [black] (DN21) circle (2pt);
			\filldraw [black] (DN1) circle (2pt);
			
			\draw (DN1) -- (DN21) -- (DN3) -- (DN4) -- (DN51) -- (DN6);
			\draw (DN1) -- (DN22) -- (DN3);
			\draw (DN1) -- (DN24) -- (DN3);
			\draw (DN1) -- (DN25) -- (DN3);
			\draw (DN4) -- (DN52) -- (DN6);
			\draw (DN4) -- (DN54) -- (DN6);
			\draw (DN4) -- (DN55) -- (DN6);
		
			\node (DS6)  at (\dsp,    \hd) {};
			\node (DS55) at (\dsp+1.2,\hd*3/4) {};
			\node (DS54) at (\dsp+0.6,\hd*3/4) {};
			\node (DS53) at (\dsp,    \hd*3/4) {\tiny$\cdots$};
			\node (DS52) at (\dsp-0.6,\hd*3/4) {};
			\node (DS51) at (\dsp-1.2,\hd*3/4) {};
			\node (DS4)  at (\dsp,    \hd*2/4) {};
			\node (DS25) at (\dsp+1.2,\hd*1/4) {};
			\node (DS24) at (\dsp+0.6,\hd*1/4) {};
			\node (DS23) at (\dsp,    \hd*1/4) {\tiny$\cdots$};
			\node (DS22) at (\dsp-0.6,\hd*1/4) {};
			\node (DS21) at (\dsp-1.2,\hd*1/4) {};
			\node (DS1)  at (\dsp,    \hd*0/4) {};
			\node (DSN)  at (\dsp,    \hd+1) {$\mathfrak{n}_{3,2}$};
			
			\filldraw [black] (DS6)  circle (2pt);
			\filldraw [black] (DS55) circle (2pt);
			\filldraw [black] (DS54) circle (2pt);
			\filldraw [black] (DS52) circle (2pt);
			\filldraw [black] (DS51) circle (2pt);
			\filldraw [black] (DS4)  circle (2pt);
			\filldraw [black] (DS25) circle (2pt);
			\filldraw [black] (DS24) circle (2pt);
			\filldraw [black] (DS22) circle (2pt);
			\filldraw [black] (DS21) circle (2pt);
			\filldraw [black] (DS1)  circle (2pt);
			
			\draw (DS1) -- (DS21) -- (DS4) -- (DS51) -- (DS6);
			\draw (DS1) -- (DS22) -- (DS4);
			\draw (DS1) -- (DS24) -- (DS4);
			\draw (DS1) -- (DS25) -- (DS4);
			\draw (DS4) -- (DS52) -- (DS6);
			\draw (DS4) -- (DS54) -- (DS6);
			\draw (DS4) -- (DS55) -- (DS6);
		
			\node (DR4)  at (\drp,    \hd) {};
			\node (DR3)  at (\drp,    \hd*2/3) {};
			\node (DR2)  at (\drp,    \hd*1/3) {};
			\node (DR1)  at (\drp,    0) {};
			\node (DRN)  at (\drp,    \hd+1) {$\mathfrak{s} \otimes \frac{\mathbb{K}[t]}{\langle t^3 \rangle}$};
			
			\filldraw [black] (DR4) circle (2pt);
			\filldraw [black] (DR3) circle (2pt);
			\filldraw [black] (DR2) circle (2pt);
			\filldraw [black] (DR1) circle (2pt);
			
			\draw (DR1) -- (DR2) -- (DR3) -- (DR4);
					
			\node (Z1) at (          \dbp,         0) {};
			\node (I1) at (\dbp-\anchCubo, \altoCubo) {};
			\node (D1) at (\dbp+\anchCubo, \altoCubo) {};
			\node (L1) at (          \dbp, \altoCubo*2) {};
			
			\node (Z2) at (          \dbp,\sepYCubo) {};
			\node (I2) at (\dbp-\anchCubo,\sepYCubo+\altoCubo) {};
			\node (D2) at (\dbp+\anchCubo,\sepYCubo+\altoCubo) {};
			\node (L2) at (          \dbp,\sepYCubo+\altoCubo*2) {};
			\node (DBN)  at (\dbp,    \hd+1) {$L_1$ and $L_2$};

			\filldraw [black] (Z1) circle (2pt);
			\filldraw [black] (Z2) circle (2pt);
			\filldraw [black] (I1) circle (2pt);
			\filldraw [black] (I2) circle (2pt);
			\filldraw [black] (D1) circle (2pt);
			\filldraw [black] (D2) circle (2pt);
			\filldraw [black] (L1) circle (2pt);
			\filldraw [black] (L2) circle (2pt);
			
			\draw (Z1) -- (I1) -- (L1);
			\draw (Z1) -- (D1) -- (L1);
			\draw [white,line width=0.7mm] (Z2)--(I2);
			\draw [white,line width=0.7mm] (Z2)--(D2);
			\draw (Z2) -- (I2) -- (L2);
			\draw (Z2) -- (D2) -- (L2);
			\draw (Z1) -- (Z2);
			\draw (D1) -- (D2);
			\draw (I1) -- (I2);
			\draw (L1) -- (L2);
		\end{tikzpicture}
		
		\caption{$L_1 = \mathfrak{s}\oplus  \mathfrak{s}\oplus \mathfrak{s}$ and $L_2 =  \mathfrak{s}\oplus  \mathfrak{s}\oplus  \mathbb{K}$, with $\mathfrak{s}$ simple.}
	
	\end{figure}
\begin{remark}
Self-duality of lattices of ideals is a necessary condition for being quadratic, but it is not equivalent to it. However, it can be a very useful tool to rule out the possibility of being quadratic. The lattices of ideals of the non-quadratic series $\fh_n$ and $\fl_n$ are not self-dual.
\end{remark}

\begin{remark}
Let define the split extension algebra of the 3-dimensional Heisenbeg $\mathfrak{a}=\fh_3\oplus \spa\langle d\rangle$, $\{x,y,z\}$ standard basis, so $[xy]=z$,  and $[d,x]=x$, $[d,y]=y$ and $[d,z]=2z$. This algebra is solvable and centreless. So, from item f) in Proposition \ref{prop:pattern-1}, it is not quadratic. But, $\mathfrak{a}^2=\fh_3$ is the unique maximal ideal and $\mathbb{K}z$ is the unique minimal. The rest of the ideals are of the form $\spa\langle z, \lambda x+\mu y\rangle$. So, the lattice of ideals of $\fa$ is self-dual.
\end{remark}
	

\section{Local quadratic algebras}\label{sec:local}

A \emph{local algebra} is an algebra with only one maximal ideal. Along this section, $\fs$, $\fr$ and $\fn$ will denote the Levi subalgebra of $\fg$, its solvable radical and the nilradical. We have the Levi decomposition $\fg=\fs\oplus \fr$. In case $\fs\neq 0\neq \fr$, we say thet $\fg$ mixed. The Jacobson radical $\mathfrak{J}(\fg)$ of $\fg$ is the solvable radical of $\fg^2$. This ideal is nilpotent and,
\begin{equation*}
    \mathfrak{J}(\fg)=[\fg,\fg]\cap \fr=[\fg,\fr]\subseteq \fn.
 \end{equation*}
When $\fg$ is solvable, $\mathfrak{J}(\fg)=\fg^2$. It is known that $\mathfrak{J}(\fg)$ is the intersection of all maximal ideals. So, in any local algebra, this radical is the only maximal ideal.

\begin{lemma}\label{lem:condicion-local} For a Lie algebra $\fg$ of dimension greater than one and not simple, the following is equivalent:
    \begin{enumerate}[\quad a)]
        \item $\fg$ is local,
        \item $\mathfrak{J}(\fg)=\fn$,
        \item $\fg=\fs \oplus \fn$ with $\fs$ one dimensional or simple and $$\fn^2=\{x\in \fn: [y,\fn]\subseteq \fn^2\ \forall y\in \fs\}.$$
    \end{enumerate}
 
\end{lemma}

\begin{remark}
Assertion $\fn^2=\{x\in \fn: [y,\fn]\subseteq \fn^2\ \forall y\in \fs\}$ in Lemma \ref{lem:condicion-local}, is equivalent to saying that $\ad x$ is faithful on $\fn/\fn^2$ for all $x\in \fn\notin \fn^2$ if $\fs$ is one-dimensional and, in the simple case, the $\ad \fs$-module $\fn/\fn^2$ has non-trivial modules.
\end{remark}

We start revisiting Theorem 3.1 in \cite{bajo2007lie}.  

\begin{proposition}[Bajo, Benayadi, 2007]\label{prop:benayadi-bajo}
   Up to isometric isomorphisms, any local and quadratic Lie algebra $(\fg, \varphi)$ is of one of the following types:
    \begin{enumerate}[\quad a)]
        \item $\fg=\spa\langle x \rangle$ and  $\varphi(x,x)=1$.
        
        \item $\fg$ is simple and there is $\lambda\neq 0$ such that $\lambda\varphi(x,y)=\tr(\ad\, x\circ \ad\, y )$ is the Killing form.
        
        \item $\fg=\fs\oplus\fs^*$ the central split extension of a simple Lie algebra $\fs$ and its dual module $\fs^*$ and $\varphi(x+\alpha, y+\beta)=\alpha(y)+\beta(x)$.
        
        \item $\fg$ is a solvable Lie algebra double extension of a nonzero and nilpotent quadratic Lie algebra $(\fa, \psi)$ by a $\psi$-skew derivation $\delta$ which is invertible on the centre $Z(\fa)$.
        
         \item $\fg=\fg^2$ is a mixed Lie algebra double extension of a nonzero and nilpotent quadratic Lie algebra $(\fa, \psi)$ by a simple subalgebra $\fs$ through a representation $\psi\colon \fs\to \der_\varphi(\fa)$ such that $Z(\fa)$ has non trivial $\fs$-modules.
        
        
        
        
    
    \end{enumerate}
    The Jacobson radical $\mathfrak{J}(\fg)=\fn$ is non-abelian for any $\fg$ in items d) or e),  $\mathfrak{J}(\fg)=\fs^*$ in item c) and null for simple and one-dimensional algebras.

\end{proposition}
\begin{proof}If $\fg=\fs\oplus\fn$ is reductive, the only possibilities are a) or b). Assume then $\fg$ non-reductive. From \cite[Theorem 3.1, item i)]{bajo2007lie} we arrive at item d) and from \cite[Theorem 3.1,item ii)]{bajo2007lie} $\fn^\perp=Z(\fn)$ we get either b) if $\fn=\fn^\perp$ or e) otherwise. Note that if $\fg$ solvable, $0\neq Z(\fg)\subseteq \fn$. Then $\fg^2=\fn=[x,\fn]\oplus \fn^2$ and $\fn^2=0$ give us $\ad\, x|_\fn$ bijective, which is not possible.
\end{proof}

\noindent From previous proposition we have the following conclusions:
    \begin{itemize}
    \item Simple and one-dimensional algebras are the only local quadratic and reductive Lie algebras (types a) and b)).
    \item Trivial $T^*$-extensions of simple Lie algebras are local and quadratic (type c)). Recall that the $T^*$-extension is a more general construction. The trivial extension is the easier case.
   \end{itemize} 
Quadratic local algebras $(\fg, \varphi)$ of type d) or e) have some common structural patterns and remarkable differences. Using Proposition \ref{prop:benayadi-bajo} and \cite[Theorem 2.2]{bordemann1997nondegenerate}, we can easily deduce the following:
    \begin{itemize}
    \item Common patterns: $\fr=\fn$ is the unique maximal ideal, $\fn^\perp$ is the unique minimal ideal and $\fn\neq \fn^\perp$. In addition, $0\neq \fn^\perp\subseteq \fn^2$ and $\fg$ can be built as a double extension of the quadratic nilpotent algebra,
    
        \begin{equation*}
        \left(\frac{\fn}{\fn^\perp},\varphi|_{\frac{\fn}{\fn^\perp}}\right).
        \end{equation*}
    
    \item Extra patterns for $\fg$ of type d): $\fg$ is solvable ($\fs=0$) and $\fn=\fg^2$ is of codimension $1$, $\fn^\perp=Z(\fg)=\spa\langle z\rangle\subseteq \fn^2\subset \fn$ and the double extension is induced through the $\varphi$-skew $\ad\, x$ for any $x\in \fg$ such that $\varphi(x,z)\neq 0$ and $\ad\, x|_{Z(\frac{\fn}{\fn^\perp})}$ is invertible.
    
     \item Extra patterns for $\fg$ of type e): $\fg=\fs\oplus\fn$ is a perfect mixed algebra and $\fs$ is simple, $ \fn^\perp=Z(\fn)\subseteq\fn^2\subset\fn=[\fg,\fn]$, $Z(\fn)\cong \fs^*$ as $\ad\, \fs$-modules and the second term of the lower central series $Z_2(\fn)$ has no trivial $\ad\, \fs$-modules. Here the double extension is induced by the adjoint representation of the Levi subalgebra, $\ad_{\frac{\fn}{\fn^\perp}}\colon \fs\to \der_\varphi(\frac{\fn}{\fn^\perp})$.
     \end{itemize}
        
\noindent In both cases, by imposing $\fn^2=\fn^\perp$, the algebra $\fg$ is a double extension of the abelian algebra $\frac{\fn}{\fn^\perp}$ and yields to the constructions given in Examples \ref{ex:oscillator} ($\fg$ solvable) and \ref{ex:ext-simple-by-abelian} ($\fg$ perfect). Otherwise, $\fn^2\neq\fn^\perp$ and local quadratic Lie algebras follow from nilpotent and nonabelian quadratic algebras. In our Examples \ref{ex:ext-by-nonabelian-1} and \ref{ex:ext-by-nonabelian-2} we will present some algebras of low nilindex.

\begin{example}
(Generalized oscillator algebras) Abelian Lie algebras are nothing else but a pair $(\fv_n, \varphi)$ where $\fv_n$ is a $n$-dimensional vector space endowed with a symmetric and non-degenerate symmetric form $\varphi$. Consider now any $\varphi$-skew linear automorphism $\delta$ of $\fv_n$. This is only possible in the case $n=2m$. Then, the double extension $\fd(2m)=\delta\oplus \fv_{2m}\oplus \delta^*$ is local and quadratic. Note that $\fh=\fd(2m)^2=\fv_{2m}\oplus \delta^*$ and its square is just $[\fd(2m)^2,\fd(2m)^2]=\delta^*$. So $\fh=\fh_m$ is a generalized Heisenberg algebra and $\ad_{\fd(2m)}\delta\in \inner \fd(2m)$ extends $\delta$ by declaring $[\delta,\delta^*]=0$. Note that $\ad_{\fh}\delta\in \der \fh$ is an outer derivation of $\fh$. \hfill $\square$
 \end{example}

\begin{remark}
Over the reals and starting with a positive definite form $\varphi$, since $\delta$ is $\varphi$-skew, there is a $\varphi$-orthonormal basis $\{e_1,\dots e_{2m}\}$ of the real space $\fv_{2m}$ such that $[e_{2i-1}, e_{2i}]_{\fd(2m)}=\lambda_ie_{2m+1}$, $e_{2m+1}=\delta^*$  for some nonzero real numbers $0<\lambda_1\leq \lambda_2 \dots \leq \lambda_m$ and $\delta(e_{2i-1})=\lambda_ie_{2i}$ and $\delta(e_{2i})=-\lambda_ie_{2i-1}$ and $\delta(e_{2m+1})=0$. According to \cite[Theorem 4.1, section 4]{medina1985groupes}, the algebras $\fg(\lambda)=\spa\langle e_0(=\delta),  e_1,\dots e_{2m}, e_{2m+1}(=\delta^*)\rangle$, $\lambda=(\lambda_1, \dots, \lambda_m)$ are named oscillator algebras. This family determines the set of non simple Lie groups, connected and simply connected, endowed with a Lorentz-invariant metric that makes them indecomposable. The series of oscillator algebras also appeared in the characterizacion of Lorentzian cones given by by Hilgert and Hofmann   \cite{hilgert1985lorentzian}.
\end{remark}

\begin{example}\label{ex:ext-simple-by-abelian}
Let $\fs$ be a simple Lie algebra and $(\fv, \varphi)$ an $\fs$-module without trivial submodules. Denote by $\rho$ the representation of $\fs$ on $\fv$. Assume that $\varphi$ is a symmetric, nondegenerate and $\fs$-invariant bilinear form. The existence of $\varphi$ only is posible if the $2$-symmetric tensor power, $S^2\fv$, has a trivial submodule. Then, the representation $\rho\colon\fs\to \fgl(\fv)$ satisfies
\begin{equation*}
    \varphi(\rho(s)(u),v))+\varphi(u,\rho(s)(v))=0\ \forall s\in \fs.
\end{equation*}
Equivalently, $\rho\colon \fs\to \der_\varphi(\fv, \varphi)$. 
Let $\fv(\fs)=\fs\oplus_\psi(\fv \oplus_\omega \fs^*)$ the double extension of $(\fv, \varphi)$ by $\rho$. The bracket and the bilinear form are given by equations \eqref{eq:de-bracket} and \eqref{eq:de-bilinear} where $\fg=\fs$, $\fa=\fv$ is an abelian algebra and $\omega(u,v)(s)=\varphi(\rho(s)(u), v)$, $u,v\in \fv$ and $s\in \fs$. Taking $\fs=\fsl_2(\ku)$ and $\fv=V(n)$, the unique possibility is $n=2m$ and the algebra $\fa(\fsl_2(\ku))$ is a ($2m+7$)-dimensional local algebra with $4$ ideals in chain. \hfill $\square$
\end{example}

In \cite{del2012free}, the authors compute $\der_\varphi \fn_{2,3}$ and $\der_\phi \fn_{3,2}$. Skew-derivations of quadratic are a natural source to get examples of local quadratic algebras by double extension. 

\begin{example}\label{ex:ext-by-nonabelian-1} Let $(\fn_{2,3}, \varphi)$ the algebra described in Example \ref{ex:ndt-free} with (Hall) basis $\{a_i\}_{i=1}^5$. Then, $\der_\varphi \fn_{2,3}$ is $6$-dimensional and it is given as the direct sum of the (nipotent) ideal $\inner \fn_{2,3}$, and its Levi subalgebra $\fs\cong \fsl_2(\ku)$. Thus, any derivation decomposes as the sum (for short, $D=D(m_i) \in \fs$ and $d=d(v_i)\in \inner  \fn_{2,3}$):
\begin{equation*}
D+d=
\begin{array}{|>{\centering\arraybackslash$} p{0.75cm} <{$}>{\centering\arraybackslash$} p{0.75cm} <{$}|>{\centering\arraybackslash$} p{0.75cm} <{$}|>{\centering\arraybackslash$} p{0.75cm} <{$}>{\centering\arraybackslash$} p{0.75cm} <{$}|}
\hline
 m_1 & m_2 & 0 & 0 & 0 \\
 m_3 & -m_1 & 0 & 0 & 0 \\
\hline
 0 & 0 & 0 & 0 & 0 \\
\hline
 0 & 0 & 0 & m_1 & m_2 \\
 0 & 0 & 0 & m_3 & -m_1 \\
\hline
\end{array}
\enspace+\enspace
\begin{array}{|>{\centering\arraybackslash$} p{0.75cm} <{$}>{\centering\arraybackslash$} p{0.75cm} <{$}|>{\centering\arraybackslash$} p{0.75cm} <{$}|>{\centering\arraybackslash$} p{0.75cm} <{$}>{\centering\arraybackslash$} p{0.75cm} <{$}|}
\hline
 0 & 0 & 0 & 0 & 0 \\
 0 & 0 & 0 & 0 & 0 \\
\hline
 v_2 & v_1 & 0 & 0 & 0 \\
\hline
 v_3 & 0 & v_1 & 0 & 0 \\
 0 & v_3 & -v_2 & 0 & 0 \\
\hline
\end{array}
\end{equation*}

\noindent Note that $\fn_{2,3}$ has two copies of the $2$-dimensional natural $\fs$-module $V(1)$ ($\langle a_1,a_2\rangle$ and $\langle a_4,a_5\rangle$) and one of the trivial module $V(0)=\langle a_3\rangle$. By double extension formulae \eqref{eq:de-bracket} and \eqref{eq:de-bilinear} we get the $11$-dimensional local (apply Lemma \ref{lem:condicion-local} or Proposition \ref{prop:benayadi-bajo}) perfect quadratic Lie algebra $\fn_{2,3}(\fs)=\fs \oplus \fn_{2,3}\oplus \fs^*$. The lattice of ideals of $\fn_{2,3}(\fs)$ is a $6$-chain. Taking any derivation $\delta=D(m_1,m_2,m_3)$ such that $m_3m_2-m_1^2\neq 0$ we get a 7-dimensional solvable local quadratic via the double extension by $\langle \delta \rangle$. \hfill $\square$
\end{example}

\begin{example}\label{ex:ext-by-nonabelian-2}
Consider now the quadratic $(\fn_{3,2}, \phi)$ in Example \ref{ex:ndt-free} with (Hall) basis $\{a_i\}_{i=1}^6$. In this case, $\der_\varphi \fn_{3,2}$ is $10$-dimensional and it is the direct sum of $\inner \fn_{3,2}$ and a simple Lie subalgebra $\fs\cong \fsl_3(\ku)$ ($D=D(m_1,\dots, m_5) \in \fs$ and $d=d(v_1,v_2,v_3)\in \inner \fn_{3,2}$):
\begin{equation*}
D+d=
\begin{array}{|>{\centering\arraybackslash$} p{0.7cm} <{$}>{\centering\arraybackslash$} p{0.7cm} <{$}>{\centering\arraybackslash$} p{1.7cm} <{$}|>{\centering\arraybackslash$} p{1.7cm} <{$}>{\centering\arraybackslash$} p{0.8cm} <{$}>{\centering\arraybackslash$} p{0.8cm} <{$}|}
\hline
 m_1 & m_2 & m_3 & 0 & 0 & 0 \\
 m_4 & m_5 & m_6 & 0 & 0 & 0 \\
 m_7 & m_8 & -m_1-m_5 & 0 & 0 & 0 \\
\hline
 0 & 0 & 0 & m_1+m_5 & m_6 & -m_3 \\
 0 & 0 & 0 & m_8 & -m_5 & m_2 \\
 0 & 0 & 0 & -m_7 & m_4 & -m_1 \\
\hline
\end{array}
\ +\ 
\begin{array}{|>{\centering\arraybackslash$} p{0.66cm} <{$}>{\centering\arraybackslash$} p{0.66cm} <{$}>{\centering\arraybackslash$} p{0.66cm} <{$}|>{\centering\arraybackslash$} p{0.66cm} <{$}|}\hline
 0 & 0 & 0 & \\
 0 & 0 & 0 & \cdots\\
 0 & 0 & 0 & \\\hline
 v_1 & v_2 & 0 &\\
 v_3 & 0 & v_2 &\cdots\\
 0 & v_3 & -v_1 &\\\hline
\end{array}
\end{equation*}
Here $\fn_{3,2}$ decomposes as two copies of the $3$-dimensional natural $\fs$-module. By double extension, we get the $22$-dimensional local (just applying Lemma \ref{lem:condicion-local}) perfect quadratic Lie algebra $\fn_{3,2}(\fs)=\fs \oplus \fn_{3,2}\oplus \fs^*$ and its lattice of ideals is a $5$-chain.  \hfill $\square$
\end{example}

\begin{example}\label{ex:ext-by-oscillator}
According to \cite[Proposition 5.1]{figueroa1996structure}, double extensions via inner derivations provide decomposable quadratic algebras. So, any double extension of $(\fd_4, \psi)$ in Example \ref{ex:oscillator} is decomposable because of $\der_\psi \fd_4=\inner  \fd_4$. \hfill $\square$
\end{example}

\section{Final comments}

In 2014, A. Elduque and S. Benayadi classified real and complex (see~\cite[Theorems 3.16 and 4.11]{benayadi2014classification}) indecomposable mixed quadratic Lie algebras of dimension $\leq 13$. The Levi subalgebra of this type of algebras in the complex case is the $3$-dimensional split $\fsl_2(\mathbb{C})$ and over the real either $\fsl_2(\mathbb{R})$ or $\fsu_2(\mathbb{R})$. Combining irreducible modules for these algebras, smart reasoning and elementary linear tools and previous knwoledge on basic structure of quadratic Lie algebras, the authors achieve a very clean classification. In the complex case, apart from the double extension process, constructions of quadratic Lie algebras as tensor products (Example \ref{ex:tensor_product}) also appeared. The authors arrived there by using Jordan algebras of dimension $2,3,4$ and following the ideas given in 1976 by B.N. Allison \cite[Theorem 1, Section 5]{allison1976construction}. Lie algebras in the complex classification are local, so indecomposable. Most part of this classification could be recovered from previous examples seen in Section \ref{sec:local}. We also point out that their lattices of ideals are mainly $n$-chains with $n$ up to six elements.

\section*{Funding}
The authors have been partially funded by grant MTM2017-83506-C2-1-P of `Ministerio de Econom\'ia, Industria y Competitividad, Gobierno de Espa\~na' (Spain). J. Rold\'an-L\'opez has been also supported by a predoctoral research grant FPI-2018 of `Universidad de La Rioja'.


\end{document}